\documentclass[11pt, a4paper]{amsart}

\usepackage{amsmath, amssymb, amsthm}
\usepackage{mathtools}
\usepackage{enumitem}
\usepackage{hyperref}
\usepackage{tikz}
\usepackage{booktabs}

\hypersetup{
  colorlinks=true,
  linkcolor=blue!60!black,
  citecolor=green!50!black,
  urlcolor=blue!70!black
}

\newtheorem{theorem}{Theorem}[section]
\newtheorem{proposition}[theorem]{Proposition}
\newtheorem{lemma}[theorem]{Lemma}
\newtheorem{corollary}[theorem]{Corollary}
\theoremstyle{definition}

\newtheorem{example}[theorem]{Example}
\newtheorem{remark}[theorem]{Remark}

\newcommand{\ZZ}{\mathbb{Z}}

\newcommand{\cA}{\mathcal{A}}

\DeclareMathOperator{\indeg}{indeg}

\title[Coordinate projections of $c$-vectors]{%
  Coordinate projections of $c$-vectors of cluster algebras
  from the annulus}

\author{Sarah B.\ Brodsky}


\date{\today}

\begin{document}

\begin{abstract}
For an acyclic cluster algebra, the $c$-vectors are, up to sign, the real
Schur roots of the associated root system. We study the two-coordinate
projections $(c_v, c_w)$ of this configuration: when the difference
$c_v - c_w$ is bounded, the image lies in a finite band of lattice lines,
and we ask when the projection fills that band. In finite type, boundedness
is automatic; in affine type we prove that it is equivalent to equality of
the corresponding null-root coordinates, $\delta_v = \delta_w$. We then
resolve the filling problem for affine type $\widetilde{A}_n$ in the
source-sink orientation: every coordinate projection fills its band except
for the source-sink pair, whose diagonal contains only the finite regular
part. More generally, we classify the non-filling banded pairs for every
acyclic orientation of every affine diagram: apart from the source-sink
diagonals in type $\widetilde{A}$, non-filling occurs only at boundary lines
reached by a unique extremal real root, and precisely when the geodesic
joining the two vertices is balanced. The obstruction is the
Auslander--Reiten defect; on the annulus this picture is topological: the
defect is the crossing number with the core curve, and the $\delta$-shift is
the Dehn twist along it.
\end{abstract}

\maketitle

\section{Introduction}
\label{sec:intro}

Let $Q$ be an acyclic quiver on vertex set $Q_0$, and let $\cA(Q)$ be the
cluster algebra with principal coefficients at the seed associated with
$Q$. The $c$-vectors of $\cA(Q)$, collected over all seeds, form a finite
or infinite subset $C(Q) \subset \ZZ^{Q_0}$; by sign-coherence
\cite{GHKK2018} each $c$-vector is either nonnegative or nonpositive, and
by N\'ajera Ch\'avez \cite{NajeraChavez2013} the positive $c$-vectors are
exactly the real Schur roots of $Q$, the dimension vectors of the
exceptional indecomposable representations. Thus
\[
  C(Q) \;=\; \{\, \pm\beta : \beta \text{ a real Schur root of } Q \,\}.
\]

For a pair of vertices $v, w \in Q_0$ we consider the coordinate
projection
\[
  \pi_{vw} \colon \ZZ^{Q_0} \to \ZZ^2,
  \qquad
  \pi_{vw}(c) = (c_v, c_w).
\]
When the difference $c_v - c_w$ is bounded over $C(Q)$, the image
$\pi_{vw}(C(Q))$ is contained in a finite union of lattice lines, the
\emph{band} $\{(x, y) : |x - y| \leq b_{vw}\}$ with $b_{vw} = \sup_{c}|c_v
- c_w|$. In finite and affine type, the present paper answers two questions: for
which pairs is the difference bounded, and, when it is, when does
$\pi_{vw}$ surject onto the band.

Two things make these questions natural. First, the $c$-vector set is a
fundamental mutation invariant: it records the tropical dynamics of the
coefficients \cite{FominZelevinsky2007, NakanishiZelevinsky2012}, it is
sign-coherent \cite{GHKK2018}, and in the acyclic case it is the set of
signed real Schur roots \cite{NajeraChavez2013, SpeyerThomas2013}, so
questions about its geometry are questions about the root system seen
through the cluster structure. Coordinate projections are the simplest
two-dimensional shadows of this configuration; the band is exactly the
region a bounded-difference shadow can occupy, and filling is the
completeness question for the shadow. Second, the answer is not merely
combinatorial: we show that the sole obstruction to filling is the
Auslander--Reiten defect, so an elementary projection statistic detects a
homological invariant; and the failure locus, which we classify completely
over every affine type and every acyclic orientation
(Section~\ref{sec:classification}), is governed by a balance law on the
geodesic joining the two coordinates. In the annulus the whole story has a
topological reading, in the line of the descriptions of $c$-vectors by
curves on surfaces \cite{FominShapiroThurston2008, Hong2021,
LeeLeeMills2019}: the $c$-vectors are the arcs of the annulus, the defect
is the crossing number with the core curve, and the $\delta$-shift is the
Dehn twist along it (Section~\ref{sec:surface}).

The first question has a uniform answer in finite and affine type; in the
affine case it is governed by the null root.

\begin{theorem}[Band-existence dichotomy]
\label{thm:band_dichotomy}
Let $Q$ be an acyclic quiver. If $Q$ is of finite type, then $C(Q)$ is
finite, so every coordinate difference is bounded. If $Q$ is of affine
type with null root $\delta$, then for $v, w \in Q_0$,
\[
  \sup_{c \in C(Q)} |c_v - c_w| < \infty
  \qquad\Longleftrightarrow\qquad
  \delta_v = \delta_w,
\]
and when $\delta_v = \delta_w$ one has
$b_{vw} = \max_{\alpha} |\alpha_v - \alpha_w|$ over the real roots
$\alpha$ of the underlying finite root system.
\end{theorem}

In type $\widetilde{A}_n$ the null root is $\delta = (1, \dots, 1)$, so
\emph{every} pair is banded, with $b_{vw} = 1$. These are the cluster
algebras of the annulus \cite{FominShapiroThurston2008}. We work throughout
with the \emph{source-sink orientation} of the $(n+1)$-cycle: an acyclic
orientation with a \emph{unique} source $s$ and a \emph{unique} sink $t$,
which divide the cycle into two oriented paths $s \to t$ of lengths $p$ and
$q$ with $p + q = n + 1$. \textup{(}Not every acyclic orientation of the
cycle is of this form: the alternating orientation has several sources and
sinks, and a different filling pattern; we return to this seed-dependence in
Remark~\ref{rem:seed_dependence}.\textup{)} In the surface model the two arcs
are the two boundary components, carrying $p$ and $q$ marked points, and the
two non-homogeneous tubes have ranks $p$ and $q$. Our main result resolves
the surjectivity question completely for this seed.

\begin{theorem}[Projection-filling]
\label{thm:main}
Let $Q$ be a source-sink orientation of the $(n+1)$-cycle, with unique source
$s$ and unique sink $t$. Then:
\begin{enumerate}[label=\textup{(\arabic*)}]
\item every coordinate difference satisfies $|c_v - c_w| \leq 1$, so the
  band is $\{|c_v - c_w| \leq 1\}$;
\item for every pair $(v, w)$ other than $(s, t)$, the projection
  $\pi_{vw}$ surjects onto the full band;
\item for the source-sink pair, $\pi_{st}(C(Q))$ is the band with the
  diagonal $\{c_s = c_t\}$ removed except for a bounded segment; the
  diagonal carries only the regular part of $C(Q)$.
\end{enumerate}
In particular the source-sink diagonal $\{c_s = c_t\}$ is the unique
non-filled band line.
\end{theorem}

The mechanism is transparent once the right invariant is identified. The
defect form of the tame hereditary algebra $kQ$, which separates
preprojective, regular, and preinjective modules, turns out to be the
\emph{coordinate difference of source and sink},
\[
  \partial(x) \;=\; x_s - x_t,
\]
a one-line consequence of the source having in-degree $0$ and the sink
in-degree $2$ in the cycle (Lemma~\ref{lem:defect}). Consequently an arc
is transjective exactly when it contains exactly one of $\{s, t\}$, and
the source-sink diagonal $\{c_s = c_t\}$ is exactly the defect-zero
locus, where only the finitely many regular real Schur roots live. Every
other line is reached by a transjective root, whose $\delta$-shifts sweep
the entire line; the realizers are supplied by two short constructions
(Section~\ref{sec:main}).

This gives the result a meaning beyond the combinatorics
(Section~\ref{sec:defect_meaning}). The defect is the canonical bounded
invariant whose zero-level is the finite set of exceptional regular roots
(Proposition~\ref{prop:defect_obstruction}). On the \emph{diagonal} this is
sharp: a banded pair has a non-filling diagonal exactly when the defect is
the coordinate difference $x_v - x_w$, which across all acyclic affine
types happens only for the source-sink pair of $\widetilde{A}_n$ in its
source-sink orientation (Corollary~\ref{cor:diagonal_classification}); there
the three band lines
of $\pi_{st}$ are the preprojective, regular, and preinjective strata
(Theorem~\ref{thm:defect_detection}). Off the diagonal the picture is
richer, and it is seed-dependent. In $\widetilde{D}_4$ the defect
$x_1 + x_2 - x_3 - x_4$ is not a coordinate difference and every coordinate
projection still fills (Example~\ref{ex:D4_meaning}); but in
$\widetilde{E}_7$ and $\widetilde{E}_8$ a banded pair can fail to fill at its
\emph{outer} band lines, which the regular exceptional $c$-vectors reach but
the transjective ones do not (Theorem~\ref{thm:E_classification}). So
non-filling pairs are not confined to $\widetilde{A}_n$, and arise from two
distinct mechanisms: an interior gap at the defect diagonal, and a boundary
extension by regular roots. Both we classify completely, over every affine
type and every acyclic orientation. The interior gap is settled by
Theorem~\ref{thm:delta_one}: every banded pair of null-root coefficient $1$
fills, except the source-sink diagonal of $\widetilde{A}_n$. Boundary
extension can occur only at coefficient at least $2$
(Lemma~\ref{lem:regular_bound}), and the classification there
(Theorems~\ref{thm:Dtilde_classification}, \ref{thm:E_classification},
and~\ref{thm:complete_classification}) takes a strikingly uniform form: the
band exceeds the transjective range only at pairs whose extremal band
lines are reached by a \emph{unique} real root, and such a pair fails to
fill precisely when the geodesic joining it is \emph{balanced}; i.e.,
carries equally many arrows in each direction. The failing pairs are the
distance-two spine pairs of $\widetilde{D}_n$ ($n \geq 6$), the pair
$(3,6)$ of $\widetilde{E}_7$, and the pairs $(1,6)$ and $(3,7)$ of
$\widetilde{E}_8$; the remaining unique-root pair, $(2,5)$ of
$\widetilde{E}_8$, has a geodesic of odd length and never fails. The
filling pattern is thus seed-dependent throughout: in type
$\widetilde{D}_n$ with $n \geq 6$ all but $32$ of the $2^{n}$ acyclic
orientations have a non-filling pair, as do $48$ of the $128$ orientations of
$\widetilde{E}_7$ and $140$ of the $256$ orientations of $\widetilde{E}_8$, while
$\widetilde{D}_4$, $\widetilde{D}_5$, and $\widetilde{E}_6$ fill
completely for every orientation.

The squarefree Pr\"ufer $4$-cycle on $\{1, 2, 3, 6\}$, an acyclic
orientation of the $4$-cycle of affine type $\widetilde{A}_3$, is the
smallest nontrivial instance; there the source-sink pair is $(1, 6)$,
and Theorem~\ref{thm:main} recovers by hand the projection structure we
work out in Section~\ref{sec:example}.

\medskip\noindent\textbf{Conventions.}\;
All quivers are finite, connected, and without loops or $2$-cycles. We
work over an algebraically closed field $k$, and identify a real root
with the dimension vector of the corresponding exceptional module. For a
subset $I \subseteq Q_0$ we write $\alpha_I = \sum_{i \in I} e_i$ for its
indicator vector.

\section{Preliminaries}
\label{sec:prelim}

In this section, we recall the facts about acyclic cluster algebras and
tame hereditary algebras that we use. We refer to
\cite{FominZelevinsky2007} for cluster algebras with principal
coefficients and $c$-vectors, to \cite{NajeraChavez2013, GHKK2018} for the
identification of $c$-vectors with real Schur roots, and to
\cite{Ringel1984, Kac1990} for the representation theory of tame
hereditary algebras and affine root systems.

\subsection*{\texorpdfstring{$c$-vectors}{c-vectors} and real Schur roots}
The $c$-vectors of $\cA(Q)$ at a seed $t$ are the columns of the $c$-matrix
$C^t$, which tracks the principal coefficients. By sign-coherence
\textup{(}\cite[Corollary~5.5]{GHKK2018}; for the skew-symmetric case,
Derksen--Weyman--Zelevinsky \cite{DWZ2010}\textup{)} each $c$-vector is
nonnegative or nonpositive, and by N\'ajera Ch\'avez
\cite[Theorem~1.3]{NajeraChavez2013} the positive $c$-vectors of an acyclic
cluster algebra are exactly the real Schur roots of $Q$, that is, the
dimension vectors of the exceptional (rigid, indecomposable) $kQ$-modules.
Hence $C(Q)$ is the set of signed real Schur roots. The $c$-vectors of affine
and tame type are thus governed by the tame hereditary representation theory
of $Q$ \cite{Ringel1984}; we use this dictionary throughout. To the best of
our knowledge the coordinate projections of the configuration $C(Q)$, and the
filling question studied here, have not been considered before.

\subsection*{Type classification}
By the Fomin--Zelevinsky classification \cite{FominZelevinsky2003II}, an
acyclic cluster algebra is of finite type if and only if the underlying
graph of $Q$ is a Dynkin diagram, and of affine (tame) type if and only
if that graph is a Euclidean diagram. The acyclic orientations of the
$(n+1)$-cycle are exactly the connected quivers of affine type
$\widetilde{A}_n$.

\subsection*{Roots of affine type}
Since $Q$ is acyclic its Euler form is unitriangular and its symmetrized
Tits form is the affine quadratic form, which is positive semidefinite
with one-dimensional radical $\ZZ\delta$ spanned by the null root
$\delta$. The real roots of a simply-laced affine root system are exactly
\[
  \{\, \alpha + m\delta : \alpha \in \Delta_{\mathrm{fin}},\ m \in \ZZ \,\},
\]
where $\Delta_{\mathrm{fin}}$ is the underlying finite root system
\cite[Ch.~6]{Kac1990}. In type $\widetilde{A}_n$ one has
$\delta = (1, \dots, 1)$, and $\Delta_{\mathrm{fin}} = A_n$; the positive
roots of $A_n$ obtained by deleting a vertex of the cycle are the
indicators of cyclic intervals, so the real roots of $\widetilde{A}_n$
are
\[
  \{\, \pm \alpha_I + m\delta : I \text{ a cyclic interval},\ m \in \ZZ \,\}.
\]

\subsection*{Preprojective, regular, preinjective}
For a tame hereditary algebra the indecomposable modules split into three
families: preprojective, regular, and preinjective. The
\emph{transjective} modules are the preprojective and the preinjective
ones; the regular modules lie in tubes, of which all but finitely many
are homogeneous (rank $1$) and the rest are non-homogeneous. In type
$\widetilde{A}_n$ there are two non-homogeneous tubes, of ranks $p$ and
$q$; in types $\widetilde{D}$ and $\widetilde{E}$ there are three. The
defect $\partial$ is the linear form on the
Grothendieck group with $\partial(M) < 0$, $= 0$, $> 0$ according as $M$
is preprojective, regular, or preinjective; a real root is regular if and
only if its defect vanishes \cite{Ringel1984}. Preprojective and
preinjective modules are directing, hence exceptional, so every
transjective real root is a real Schur root. In a tube of rank $r$ the
exceptional regular modules are exactly those of quasi-length less than
$r$, finitely many.

\section{The band-existence dichotomy}
\label{sec:band}

In this section, we prove Theorem~\ref{thm:band_dichotomy}, valid in
finite and affine type.

\begin{proof}[Proof of Theorem~\ref{thm:band_dichotomy}]
If $Q$ is of finite type, then $C(Q)$ is a finite set of roots, so every
coordinate difference is bounded.

Suppose $Q$ is of affine type. By the preliminaries every $c$-vector is
$\pm(\alpha + m\delta)$ with $\alpha \in \Delta_{\mathrm{fin}}$ and
$m \in \ZZ$, whence
\[
  c_v - c_w \;=\; \pm(\alpha_v - \alpha_w) + m\,(\delta_v - \delta_w).
\]
The set $\{\alpha_v - \alpha_w : \alpha \in \Delta_{\mathrm{fin}}\}$ is
finite. If $\delta_v = \delta_w$, the term in $m$ vanishes, so
$|c_v - c_w| \leq \max_\alpha |\alpha_v - \alpha_w| < \infty$. If
$\delta_v \neq \delta_w$, recall that the transjective real roots are
real Schur roots, hence $c$-vectors, and that there are infinitely many of
them, the preprojective and preinjective components each being infinite
\textup{(}\cite[\S3]{Ringel1984}\textup{)}. Each such root is
$\alpha + m\delta$ with $\alpha \in \Delta_{\mathrm{fin}}$, and
$\Delta_{\mathrm{fin}}$ is finite, so $|m|$ is unbounded over them. As
$\delta_v \neq \delta_w$, the difference $c_v - c_w = \pm(\alpha_v -
\alpha_w) + m(\delta_v - \delta_w)$ has bounded $\alpha$-part and unbounded
$m$-term, hence is unbounded. Finally the displayed bound is sharp: after
replacing a finite root by its negative, write it as a positive finite root
$\alpha$; the representatives $\alpha$ and $\delta-\alpha$ both lie
componentwise between $0$ and $\delta$, and are, up to sign, real Schur
roots, and hence $c$-vectors. Since $\delta_v=\delta_w$, they have
opposite coordinate differences, so a root attaining the displayed maximum
realizes $b_{vw}$. This proves the equivalence and the stated value of
$b_{vw}$.
\end{proof}

\begin{example}[The dichotomy in type $\widetilde{D}_4$]
\label{ex:D4}
Let $Q$ be the $4$-subspace quiver: a central vertex $0$ and four leaves
$1, 2, 3, 4$, with any acyclic orientation of the four edges. The null
root is $\delta = (2, 1, 1, 1, 1)$, with the central coordinate equal to
$2$. By Theorem~\ref{thm:band_dichotomy}, the six leaf-leaf pairs are
banded, while the four center-leaf pairs are not. A direct enumeration of
$c$-vectors confirms this: the leaf-leaf differences never exceed $1$,
whereas each center-leaf difference takes arbitrarily large values.
\end{example}

\section{The defect form in type \texorpdfstring{$\widetilde{A}_n$}{An}}
\label{sec:defect}

In this section, we identify the defect form of an acyclic cycle quiver
and deduce the transjective criterion. Fix a source-sink orientation $Q$ of
the $(n+1)$-cycle, with unique source $s$ and unique sink $t$.

\begin{lemma}[Defect and transjective criterion]
\label{lem:defect}
For any acyclic orientation of the $(n+1)$-cycle, with sources $s_1, \dots,
s_k$ and sinks $t_1, \dots, t_k$, the defect form of $kQ$ is, up to sign,
\[
  \partial(x) \;=\; \sum_{i=1}^{k} x_{s_i} - \sum_{i=1}^{k} x_{t_i};
\]
this is a coordinate difference if and only if $k = 1$, that is, the source
and sink are unique, in which case $\partial(x) = x_s - x_t$. Under that
\textup{(}standing\textup{)} hypothesis a cyclic interval indicator
$\alpha_I$ is regular if and only
if $I$ contains both or neither of $\{s, t\}$, and transjective if and
only if $I$ contains exactly one of $\{s, t\}$. Moreover, if $\alpha_I$ is
transjective, then $\alpha_I + m\delta$ is a $c$-vector for every
$m \in \ZZ$, whereas the regular part of $C(Q)$ is finite.
\end{lemma}

\begin{proof}
The defect is $\partial(x) = \langle \delta, x \rangle$, where
$\langle x, y \rangle = \sum_{i} x_i y_i - \sum_{i \to j} x_i y_j$ is the
Euler form. Since $\delta = (1, \dots, 1)$,
\[
  \partial(x)
  \;=\; \sum_i x_i - \sum_{i \to j} x_j
  \;=\; \sum_i x_i - \sum_{j} \indeg(j)\, x_j.
\]
In the $(n+1)$-cycle every vertex has degree $2$, so a source has in-degree
$0$, a sink in-degree $2$, and every other vertex in-degree $1$; hence
$\sum_j \indeg(j)\, x_j = \sum_j x_j + \sum_i x_{t_i} - \sum_i x_{s_i}$, and
\[
  \partial(x) \;=\; \sum_i x_i - \Bigl(\sum_j x_j + \sum_i x_{t_i}
  - \sum_i x_{s_i}\Bigr) \;=\; \sum_i x_{s_i} - \sum_i x_{t_i}.
\]
This has $2k$ nonzero coefficients, so it is a coordinate difference exactly
when $k = 1$, giving $\partial(x) = x_s - x_t$; we assume this from now on.
A real root is regular if and only if its defect vanishes
\cite{Ringel1984}; for a $0/1$ indicator $\alpha_I$ this means
$(\alpha_I)_s = (\alpha_I)_t$, that is, $I$ contains both or neither of
$\{s, t\}$. Otherwise $\partial(\alpha_I) = \pm 1$ and $\alpha_I$ is
transjective, which happens exactly when $I$ contains one of $\{s, t\}$.

For the last statement, the defect is unchanged by adding $\delta$, since
$\partial(\delta) = \delta_s - \delta_t = 0$. Thus if $\alpha_I$ is
transjective, then $\alpha_I + m\delta$ has nonzero defect for every $m$,
so it is the dimension vector of a preprojective or preinjective
indecomposable, hence exceptional, hence a $c$-vector. The regular real
Schur roots are the exceptional regular modules; each of the two
non-homogeneous tubes contributes finitely many, and the homogeneous
tubes contain no exceptional modules, so the regular part of $C(Q)$ is
finite, as desired.
\end{proof}

\section{The projection-filling theorem}
\label{sec:main}

In this section, we prove Theorem~\ref{thm:main}. We keep $Q$, $s$, $t$
as above, and we say a band line $\{c_v - c_w = d\}$ is \emph{filled} if
$\pi_{vw}(C(Q))$ contains every integer point of it.

We first record the elementary but decisive observation that a single
transjective root fills its whole line.

\begin{lemma}[Transjective roots fill]
\label{lem:fills}
Let $\alpha_I$ be a transjective cyclic interval and set
$d = (\alpha_I)_v - (\alpha_I)_w$. Then the line $\{c_v - c_w = d\}$ is
filled.
\end{lemma}

\begin{proof}
By Lemma~\ref{lem:defect}, $\alpha_I + m\delta \in C(Q)$ for every
$m \in \ZZ$. Its projection is
$\pi_{vw}(\alpha_I + m\delta) = ((\alpha_I)_v + m,\, (\alpha_I)_w + m)$,
which runs over every integer point of $\{c_v - c_w = d\}$ as $m$ varies, as
desired.
\end{proof}

Observe that, by applying the sign $-1$, a transjective root with
difference $d$ also fills the line $\{c_v - c_w = -d\}$. We now produce
transjective realizers for every line except the source-sink diagonal.

\begin{proof}[Proof of Theorem~\ref{thm:main}]
\textbf{(1)}\; By Theorem~\ref{thm:band_dichotomy} with
$\delta = (1, \dots, 1)$, every difference is bounded, and the bound is
$\max_\alpha |\alpha_v - \alpha_w| = 1$ since the finite roots of $A_n$
are signed interval indicators.

\textbf{(2) and the filled lines.}\; Fix a pair $(v, w)$ and a value
$d \in \{-1, 0, 1\}$. By Lemma~\ref{lem:fills} it suffices to exhibit a
transjective cyclic interval $I$ with $(\alpha_I)_v - (\alpha_I)_w = d$,
that is, by Lemma~\ref{lem:defect}, an interval containing exactly one of
$\{s, t\}$ with the prescribed relation to $v, w$.

\emph{The lines $d = \pm 1$.} We treat $d = +1$, where the interval must
contain $v$ and not $w$; the case $d = -1$ is symmetric. If
$v \in \{s, t\}$, the singleton $\{v\}$ contains exactly one of
$\{s, t\}$ and excludes $w \neq v$, as required. If $v \notin \{s, t\}$,
let $I_s$ be the unique arc from $v$ to $s$ avoiding $t$, and $I_t$ the
unique arc from $v$ to $t$ avoiding $s$; each exists because deleting the
avoided vertex opens the cycle into a path containing both endpoints.
Then $I_s$ contains $s$ but not $t$, and $I_t$ contains $t$ but not $s$,
so both are transjective, and both contain $v$. They leave $v$ toward
different neighbors: if they shared a first edge, then walking from $v$
along it we would meet one of $s, t$ before the other, and the arc avoiding
that vertex could not have used the edge. Hence $I_s \cap I_t = \{v\}$.
Since $w \neq v$, at least one of $I_s, I_t$ excludes $w$, and that interval
is the desired realizer.

\emph{The lines $d = 0$.} Here the interval must contain both or neither
of $\{v, w\}$. The singleton $\{s\}$ is transjective, and it contains
neither $v$ nor $w$ precisely when $s \notin \{v, w\}$; likewise $\{t\}$
works precisely when $t \notin \{v, w\}$. If $\{v, w\} \neq \{s, t\}$,
then at least one of $s, t$ lies outside $\{v, w\}$, so the corresponding
singleton fills the line $d = 0$.

Thus for every pair $(v, w) \neq (s, t)$ all three lines $d \in
\{-1, 0, 1\}$ are filled, proving (2); and for the pair $(s, t)$ the
lines $d = \pm 1$ are filled by $\{s\}$ and $\{t\}$.

\textbf{(3) The source-sink diagonal.}\; For the pair $(s, t)$ and
$d = 0$, the condition $(\alpha_I)_s = (\alpha_I)_t$ is the defect-zero
condition, so by Lemma~\ref{lem:defect} no transjective interval realizes
it; only regular roots lie on the diagonal $\{c_s = c_t\}$. The regular
part of $C(Q)$ is finite (Lemma~\ref{lem:defect}), so the diagonal
carries only finitely many $c$-vectors and is not filled. This is the
unique non-filled band line, as desired.
\end{proof}

\begin{remark}[Invariant phrasing]
\label{rem:invariant}
The exceptional locus is intrinsic: it is the defect-zero line of the
source-sink pair,
\[
  \{c_s = c_t\} \;=\; \{\partial = 0\}.
\]
Equivalently, it is the projection direction along which the two tubes,
and only the tubes, are seen. Theorem~\ref{thm:main} thus reads: \emph{the
$c$-vector set of an annulus cluster algebra surjects onto the full
bounded-difference band in every coordinate pair, except along the
defect-zero line of the source-sink pair, where only its regular part
lives.}
\end{remark}

\section{The squarefree Pr\"ufer \texorpdfstring{$4$}{4}-cycle}
\label{sec:example}

In this section, we work out the smallest nontrivial case, the
$\widetilde{A}_3$ cluster algebra of the $4$-cycle, in the orientation
arising from the divisibility order on the squarefree integers with prime
factors in $\{2, 3\}$.

\begin{example}[$\widetilde{A}_{2,2}$]
\label{ex:prufer}
Let $V = \{1, 2, 3, 6\}$, with arrows recording multiplication by a
prime,
\[
  1 \to 2, \qquad 1 \to 3, \qquad 2 \to 6, \qquad 3 \to 6.
\]
The underlying graph is the $4$-cycle in the cyclic order
$(1, 2, 6, 3)$, of affine type $\widetilde{A}_3$. The source is $s = 1$
and the sink is $t = 6$, dividing the cycle into the two arcs $1 \to 2
\to 6$ and $1 \to 3 \to 6$, each of length $2$; the two tubes have rank
$2$. By Lemma~\ref{lem:defect} the defect is $\partial = (\cdot)_1 -
(\cdot)_6$.

The cyclic intervals containing exactly one of $\{1, 6\}$, hence the
transjective indicators, are
\[
  \{1\},\ \{6\},\ \{1, 2\},\ \{2, 6\},\ \{6, 3\},\ \{3, 1\},\
  \{2, 6, 3\},\ \{1, 2, 3\},
\]
while those containing both or neither, the regular indicators, are
\[
  \{2\},\ \{3\},\ \{1, 2, 6\},\ \{1, 3, 6\}.
\]
The four regular indicators are the quasi-simples $e_2, e_3, \delta -
e_3, \delta - e_2$ of the two rank-$2$ tubes. By
Theorem~\ref{thm:main}, every coordinate projection surjects onto
$\{|c_v - c_w| \leq 1\}$ except $\pi_{16}$, whose diagonal $\{c_1 = c_6\}$
is bounded: the only $c$-vectors on it are the four quasi-simples and their
negatives, which project to $c_1 = c_6 \in \{-1, 0, 1\}$.
\end{example}

Figure~\ref{fig:proj} illustrates the two qualitatively different
projections.

\begin{figure}[ht]
\centering
\begin{tikzpicture}[scale=0.62, >=stealth]
  \tikzset{bandline/.style={gray!55, very thin}}
  \begin{scope}
    \draw[bandline] (-2.4,-3.4) -- (3.6,2.6);
    \draw[bandline] (-3.4,-3.4) -- (3.6,3.6);
    \draw[bandline] (-3.4,-2.4) -- (2.6,3.6);
    \draw[->] (-3.4,0) -- (3.8,0) node[right] {\small $c_2$};
    \draw[->] (0,-3.4) -- (0,3.8) node[above] {\small $c_3$};
    \foreach \x in {-3,...,3} {
      \fill (\x,\x) circle (1.7pt);
      \pgfmathtruncatemacro{\xp}{\x+1}
      \ifnum\xp<4 \fill (\xp,\x) circle (1.7pt);\fi
      \pgfmathtruncatemacro{\xm}{\x-1}
      \ifnum\xm>-4 \fill (\xm,\x) circle (1.7pt);\fi
    }
    \node at (0,-4.3) {\small $\pi_{23}$: full band};
  \end{scope}
  \begin{scope}[xshift=10cm]
    \draw[bandline] (-2.4,-3.4) -- (3.6,2.6);
    \draw[red!70!black, dashed, semithick] (-3.4,-3.4) -- (3.6,3.6);
    \draw[bandline] (-3.4,-2.4) -- (2.6,3.6);
    \draw[->] (-3.4,0) -- (3.8,0) node[right] {\small $c_1$};
    \draw[->] (0,-3.4) -- (0,3.8) node[above] {\small $c_6$};
    \foreach \x in {-3,...,3} {
      \pgfmathtruncatemacro{\xp}{\x+1}
      \ifnum\xp<4 \fill (\xp,\x) circle (1.7pt);\fi
      \pgfmathtruncatemacro{\xm}{\x-1}
      \ifnum\xm>-4 \fill (\xm,\x) circle (1.7pt);\fi
    }
    \foreach \x in {-1,0,1} {
      \draw[red!70!black, thick, fill=white] (\x,\x) circle (2.5pt);
    }
    \node at (0,-4.3) {\small $\pi_{16}$: diagonal bounded};
  \end{scope}
\end{tikzpicture}
\caption{Projections of the $c$-vector set of $\widetilde{A}_{2,2}$; the three
band lines $c_v - c_w \in \{-1, 0, 1\}$ are drawn in gray, and a solid dot
marks a filled line. For the non-source-sink pair $(2, 3)$ all three lines are
filled. For the source-sink pair $(1, 6)$ the off-diagonal lines are filled,
but the diagonal $c_1 = c_6$ \textup{(}dashed\textup{)} carries only the
bounded regular part: the three open circles at $c_1 = c_6 \in \{-1, 0, 1\}$.}
\label{fig:proj}
\end{figure}
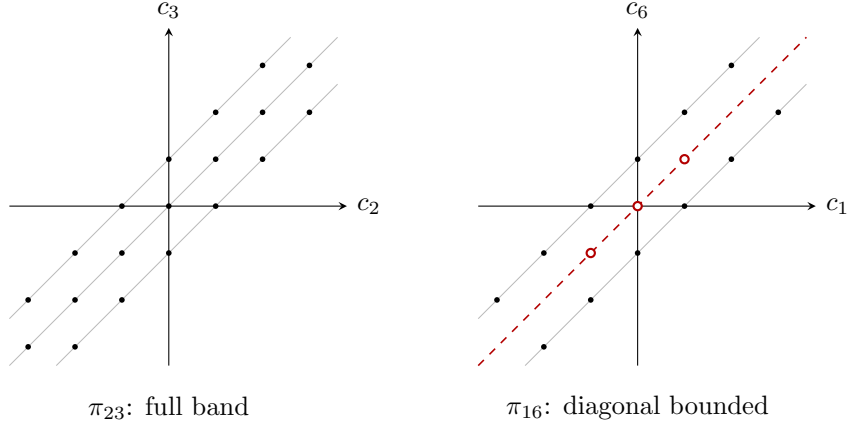

\begin{remark}[Origin]
\label{rem:origin}
The pair $(\widetilde{A}_{2,2}, (1, 6))$ is the squarefree Pr\"ufer
$4$-cycle of Example~\ref{ex:prufer}, the divisor lattice of $6$ under
divisibility, with the bottom $1$ as the source and the top $6$ as the
sink. Theorem~\ref{thm:main} explains, intrinsically, why the
source-sink direction behaves differently from the others: it is the
defect-zero line of the source-sink pair.
\end{remark}

\section{The defect as the universal obstruction}
\label{sec:defect_meaning}

In this section, we explain what projection-filling means: the obstruction
to surjectivity is, in every acyclic affine type, the Auslander--Reiten
defect, and a coordinate projection exhibits it on its \emph{diagonal}
precisely when the defect happens to be a coordinate difference. In type
$\widetilde{A}_n$ the defect is the coordinate difference
$\partial = (\cdot)_s - (\cdot)_t$, and this is why the source-sink pair
has a non-filled diagonal; in the other affine types the defect is not a
coordinate difference, and no diagonal fails. Off the diagonal the defect
resurfaces at the extremal band lines, where it governs the
boundary-extension failures classified in
Section~\ref{sec:classification}.

We first isolate the filling criterion underlying
Theorem~\ref{thm:main}, in a form valid in every acyclic affine type.

\begin{lemma}[Filling criterion]
\label{lem:filling_criterion}
Let $Q$ be acyclic of affine type and $(v, w)$ a banded pair, that is,
$\delta_v = \delta_w$. A band line $\{c_v - c_w = d\}$ is filled if and
only if some \emph{transjective} finite root $\alpha$ \textup{(}defect
$\partial(\alpha) \neq 0$\textup{)} satisfies $\alpha_v - \alpha_w = d$;
otherwise it carries only finitely many $c$-vectors.
\end{lemma}

\begin{proof}
If such $\alpha$ exists, then $\partial(\alpha + m\delta) = \partial(\alpha)
\neq 0$ for every $m$, so $\alpha + m\delta$ is a transjective real root,
hence the dimension vector of a preprojective or preinjective
indecomposable, hence exceptional, hence a $c$-vector. Its projection
$\pi_{vw}(\alpha + m\delta) = (\alpha_v + m, \alpha_w + m)$ runs over every
integer point of the line as $m$ varies, since $\delta_v = \delta_w$. So
the line is filled. Conversely, every $c$-vector is $\pm(\alpha + m\delta)$;
if no transjective $\alpha$ realizes the difference $d$, then every
$c$-vector on the line has regular finite-root part ($\partial = 0$), and
the regular real Schur roots are finite in number, so the line carries
only finitely many $c$-vectors, as desired.
\end{proof}

The role of the defect is now visible: a line is unfilled exactly when its
difference value is achieved only on the regular part, and the regular
part is finite. We record this finiteness intrinsically.

\begin{proposition}[The defect is bounded with finite zero-level]
\label{prop:defect_obstruction}
Let $Q$ be acyclic of affine type. The defect $\partial$ is bounded on
$C(Q)$. Its zero-level $\{c \in C(Q) : \partial(c) = 0\}$ is finite,
consisting of the signed exceptional regular roots, while every nonzero
level of $\partial$ on $C(Q)$ is infinite.
\end{proposition}

\begin{proof}
The defect is invariant under the Auslander--Reiten translate and is
constant on each of the finitely many transjective $\tau$-orbits, so it
takes finitely many values on $C(Q)$. Its zero-level is the set of regular
$c$-vectors, that is, the signed exceptional regular roots; these lie in the
finitely many non-homogeneous tubes and are finite in number, the
homogeneous tubes containing no exceptional modules. Each nonzero value of $\partial$ is
attained on a full transjective $\tau$-orbit, which is infinite, so the
corresponding level is infinite, as desired.
\end{proof}

Thus $\partial$ is the canonical bounded invariant whose vanishing cuts
out the finite regular part; it is the universal obstruction to filling. A
coordinate projection sees this obstruction exactly when its difference is
the defect.

\begin{theorem}[Projection-filling detects the defect]
\label{thm:defect_detection}
Let $Q$ be acyclic of affine type, and let $(v, w)$ be a banded pair with
$x_v - x_w = \pm\partial$ as linear forms. Then $\pi_{vw}$ fails to fill:
the diagonal $\{c_v - c_w = 0\}$ is the defect-zero locus, and it carries
only the finite regular part of $C(Q)$. In type $\widetilde{A}_n$ the
defect is the coordinate difference $\partial = (\cdot)_s - (\cdot)_t$,
so the hypothesis holds precisely for the source-sink pair $(s, t)$;
there $\partial$ takes only the values $-1, 0, +1$, the band is
$\{|c_s - c_t| \leq 1\}$, the three band lines $c_s - c_t = -1, 0, +1$
are exactly the preprojective, regular, and preinjective strata of
$C(Q)$, and the diagonal is the unique non-filled line. This is the
content of Theorem~\ref{thm:main}.
\end{theorem}

\begin{proof}
If $x_v - x_w = \pm\partial$, then the line $\{c_v - c_w = 0\}$ is the
defect-zero locus, which by Proposition~\ref{prop:defect_obstruction} is
finite, hence unfilled. In type $\widetilde{A}_n$,
Lemma~\ref{lem:defect} gives $\partial = (\cdot)_s - (\cdot)_t$, and
$x_v - x_w = \pm\partial$ holds for a coordinate pair if and only if
$\{v, w\} = \{s, t\}$. Every indecomposable of $\widetilde{A}_n$ has
dimension vector $\alpha_I + m\delta$ with $\alpha_I$ a $0/1$ indicator,
so $\partial(\alpha_I + m\delta) = (\alpha_I)_s - (\alpha_I)_t \in
\{-1, 0, 1\}$, the three values being the preprojective, regular, and
preinjective defects; the band is $\{|c_s - c_t| \leq 1\}$, and the
off-diagonal lines are filled by Theorem~\ref{thm:main}, as desired.
\end{proof}

The point of Theorem~\ref{thm:defect_detection} is that the elementary
combinatorial question of projection-surjectivity recovers a homological
invariant: in the annulus $\widetilde{A}_n$ the unique non-filling
coordinate pair is the one computing the Auslander--Reiten defect
(Theorem~\ref{thm:main}). Off the diagonal, and in other affine types, the
relationship is looser. The next example shows a type where the defect is
not a coordinate difference and every coordinate projection nonetheless
fills; by contrast, in types $\widetilde{E}_7$ and $\widetilde{E}_8$ (and
in $\widetilde{D}_n$ for most acyclic orientations) a projection can fail to fill off the
diagonal (Theorem~\ref{thm:E_classification}), so the defect being a
coordinate difference is \emph{not} necessary for non-filling.

\begin{example}[The obstruction without a non-filling pair: $\widetilde{D}_4$]
\label{ex:D4_meaning}
Let $Q$ be the affine quiver $\widetilde{D}_4$ with central vertex $0$ and
leaves $1, 2, 3, 4$, in the orientation $1 \to 0$, $2 \to 0$, $0 \to 3$,
$0 \to 4$. The null root is $\delta = (2, 1, 1, 1, 1)$, and the defect,
computed as in Lemma~\ref{lem:defect}, is
\[
  \partial \;=\; x_1 + x_2 - x_3 - x_4,
\]
which is \emph{not} a coordinate difference. By
Theorem~\ref{thm:band_dichotomy} the banded pairs are the six leaf-leaf
pairs, each with bound $1$. A direct enumeration of $c$-vectors shows that
\emph{every} leaf-leaf projection fills its band: there is no non-filling
coordinate pair. Yet the defect remains the obstruction in the sense of
Proposition~\ref{prop:defect_obstruction}: on $C(Q)$ it takes the values
$-2, -1, 0, 1, 2$, with the zero-level finite \textup{(}twelve signed
exceptional regular roots\textup{)} and each nonzero level infinite. The
obstruction has simply migrated off the coordinate axes: it lives in the
linear form $\partial$, which no single coordinate projection isolates
because $\partial$ is not a coordinate difference.
\end{example}

\begin{remark}[The defect, the Coxeter--Jordan invariant, and the divisor lattice]
\label{rem:coxeter_jordan}
In type $\widetilde{A}_n$ the defect $\partial = (\cdot)_s - (\cdot)_t$
also has a Coxeter-theoretic reading: it is, up to sign, the unique
linear invariant of the Coxeter transformation's size-two Jordan block at
eigenvalue $1$ \textup{(}the partner of the null root $\delta$\textup{)}.
For the divisor lattice of Example~\ref{ex:prufer}, where the
vertices are the divisors of a product of two primes, the source is the
bottom $\hat{0} = 1$ and the sink is the top $\hat{1} = 6$ of the divisor
lattice, so the defect is the difference of the extreme coordinates,
$\partial = x_{\hat 0} - x_{\hat 1}$, of multiplicative weight $\log 6 =
\log(2 \cdot 3)$. The maximally composite direction of the lattice and the
Auslander--Reiten defect of the cluster category thus coincide.
\end{remark}

\section{Non-filling pairs across affine types}
\label{sec:classification}

In this section, we go beyond the annulus and classify the non-filling
banded pairs completely, over every acyclic affine quiver and every acyclic
seed. We first settle the diagonal lines in every type
(Corollary~\ref{cor:diagonal_classification}) and, more generally, every
banded pair of null-root coefficient one (Theorem~\ref{thm:delta_one}): such
a pair fills unless it is the source-sink diagonal of $\widetilde{A}_n$. The
remaining pairs have coefficient at least $2$; for these we develop a short
calculus of box roots (Lemmas~\ref{lem:box} through~\ref{lem:flip}), which reduces
every line of every band to the defect of finitely many explicit roots, and
we prove the complete classification
(Theorems~\ref{thm:Dtilde_classification}, \ref{thm:E_classification},
and~\ref{thm:complete_classification}): beyond the source-sink diagonals,
non-filling occurs exactly at the pairs with a unique extremal root whose
geodesic is balanced. The filling pattern is thus seed-dependent:
non-filling occurs for most acyclic orientations in type
$\widetilde{D}_{n \geq 6}$ and for a positive fraction of orientations in
$\widetilde{E}_7$ and $\widetilde{E}_8$, while $\widetilde{D}_4$,
$\widetilde{D}_5$, and $\widetilde{E}_6$ fill for every orientation
(Remark~\ref{rem:generic}).

Throughout, we write $\ell = x_v - x_w$ for a banded pair $(v, w)$. The
arguments rest on the following spanning property of the hyperplane
$\ker\ell = \{x_v = x_w\}$.

\begin{lemma}[Spanning]
\label{lem:spanning}
Let $Q$ be a connected acyclic quiver of affine type, and let $(v, w)$ be a
pair of distinct vertices. Then the real roots lying on $\ker\ell$ span it.
\end{lemma}

\begin{proof}
The hyperplane $\ker\ell$ has dimension $|Q_0| - 1$, and contains the
$|Q_0| - 2$ linearly independent simple roots $e_i$ with $i \neq v, w$. It
therefore suffices to exhibit one real root on $\ker\ell$ outside their span,
that is, a real root $\alpha$ with $\alpha_v = \alpha_w \neq 0$. The
underlying graph of $Q$ is a tree or a cycle, so it contains a geodesic
$v = u_0, u_1, \dots, u_k = w$ whose vertices induce a subquiver of type
$A_{k+1}$ \textup{(}an induced path: a tree has no chords, and a proper arc
of a cycle is a path\textup{)}. The sum of its simple roots,
$\alpha = \sum_{i=0}^{k} e_{u_i}$, is the highest root of that subsystem,
hence a real root of $Q$, and satisfies $\alpha_v = \alpha_w = 1$. Together
with the $e_i$, $i \neq v, w$, it spans $\ker\ell$, as desired.
\end{proof}

\begin{proposition}[The diagonal detects the defect]
\label{prop:diagonal_defect}
The diagonal line $\{c_v - c_w = 0\}$ fills if and only if
$\ell \neq \pm\partial$. Equivalently, the diagonal fails to fill if and only
if the defect is the coordinate difference $\ell$.
\end{proposition}

\begin{proof}
By Lemma~\ref{lem:filling_criterion} the diagonal fills if and only if some
transjective root lies on $\ker\ell$. If $\ell = \pm\partial$, then
$\ker\ell = \ker\partial$ contains no transjective root \textup{(}those
have $\partial \neq 0$\textup{)}, so the diagonal does not fill. Suppose
instead $\ell \neq \pm\partial$. Were every real root on $\ker\ell$ regular,
then by Lemma~\ref{lem:spanning} the real roots spanning $\ker\ell$ would all
lie in $\ker\partial$, giving $\ker\ell \subseteq \ker\partial$; as both are
hyperplanes this is an equality, so $\partial = c\,\ell$ for a scalar $c$. But
$\ell$ is primitive, and $\partial = \langle\delta, -\rangle$ is primitive as
well \textup{(}the unimodular Euler form sends the primitive null root
$\delta$ to a primitive coefficient vector\textup{)}, so $c = \pm 1$ and
$\ell = \pm\partial$, a contradiction. Hence some real root on $\ker\ell$ is
transjective, and the diagonal fills, as desired.
\end{proof}

\begin{corollary}[The diagonal classification, all affine types]
\label{cor:diagonal_classification}
A banded coordinate pair $(v, w)$ of a connected acyclic affine quiver $Q$
has a non-filling diagonal if and only if $Q$ is a source-sink orientation of
the $(n+1)$-cycle \textup{(}type $\widetilde{A}_n$\textup{)} and
$\{v, w\} = \{s, t\}$ is its source-sink pair. In particular such a pair is
unique when it exists, and no diagonal fails to fill in types
$\widetilde{D}$, $\widetilde{E}$, or in any cycle orientation with more than
one source.
\end{corollary}

\begin{proof}
By Proposition~\ref{prop:diagonal_defect} a diagonal fails if and only if
$\ell = \pm\partial$; i.e., if and only if $\partial$ is a coordinate
difference. Now
$\partial = \sum_j \partial_j\, x_j$ with
$\partial_j = \delta_j - \sum_{i \to j}\delta_i$ is a coordinate difference
precisely when exactly two coefficients $\partial_j$ are nonzero, equal to
$+1$ and $-1$. For the cycle ($\widetilde{A}_n$), Lemma~\ref{lem:defect}
gives $\partial = \sum_i x_{s_i} - \sum_i x_{t_i}$ over the $k$ sources and
$k$ sinks; this is a coordinate difference if and only if $k = 1$, in
which case
$\partial = x_s - x_t$ is realized by the source-sink pair, while an
orientation with $k \geq 2$ has $2k \geq 4$ nonzero coefficients and no
failing diagonal. For trees we count nonzero coefficients by leaves: at a
leaf $j$ with unique neighbor $k$, the null-root relation $2\delta_j =
\sum_{i \sim j}\delta_i$ gives $\delta_k = 2\delta_j$, so $\partial_j =
\delta_j$ if $j$ is a source and $\partial_j = \delta_j - \delta_k =
-\delta_j$ if $j$ is a sink; thus $\partial_j = \pm\delta_j \neq 0$ at
\emph{every} leaf, in any orientation. Types $\widetilde{D}$ and
$\widetilde{E}$ are trees with at least three leaves \textup{(}four for
$\widetilde{D}_n$, three for $\widetilde{E}_n$\textup{)}, so $\partial$ has
at least three nonzero coefficients and is not a coordinate difference. The
stated classification follows, as desired.
\end{proof}

\begin{remark}[Seed-dependence of the filling pattern]
\label{rem:seed_dependence}
The $c$-vector set $C(Q)$, and with it the filling pattern, depends on the
chosen acyclic seed, not merely on its mutation type. On the $4$-cycle with
vertices $\{0, 1, 2, 3\}$ the source-sink orientation $0 \to 1 \to 2$,
$0 \to 3 \to 2$ has defect $x_0 - x_2$ and the unique non-filling pair
$(0, 2)$ of Theorem~\ref{thm:main}, whereas the alternating orientation
$0 \to 1 \leftarrow 2 \to 3 \leftarrow 0$ has two sources $\{0, 2\}$ and two
sinks $\{1, 3\}$, defect $\partial = x_0 - x_1 + x_2 - x_3$ \textup{(}not a
coordinate difference\textup{)}, and hence no non-filling pair at all, by
Theorem~\ref{thm:delta_one} \textup{(}the diagonal case being
Corollary~\ref{cor:diagonal_classification}\textup{)}. Both are of cluster type
$\widetilde{A}_{2, 2}$, the type of the Pr\"ufer example
\textup{(}Example~\ref{ex:prufer}\textup{)}; so non-filling is a property of
the source-sink seed, not of the annulus cluster algebra in the abstract.
This is why we fix the source-sink orientation throughout.
\end{remark}

We next bound the regular contribution to any band; it confines all further
non-filling to pairs of null-root coefficient at least $2$.

\begin{lemma}[Regular band bound]
\label{lem:regular_bound}
Let $(v, w)$ be a banded pair, so $\delta_v = \delta_w$. Every regular
$c$-vector $c$ satisfies $|c_v - c_w| \leq \delta_v$. In particular a band line
$\{c_v - c_w = d\}$ with $|d| \geq 2$ that is met by a regular $c$-vector
forces $\delta_v = \delta_w \geq 2$.
\end{lemma}

\begin{proof}
A regular $c$-vector is, up to sign, the dimension vector of an exceptional
regular module: the positive $c$-vectors are the real Schur roots
(\S\ref{sec:prelim}), a regular one is a real Schur root of vanishing defect,
and such a root is the dimension vector of an exceptional regular
indecomposable. Write the $c$-vector as $\pm\beta$ with $\beta = \underline{\dim}\,M$;
it suffices to prove $0 \leq \beta \leq \delta$ componentwise, since then both
$\beta_v$ and $\beta_w$ lie in $[0, \delta_v]$ (using $\delta_w = \delta_v$),
so $|c_v - c_w| = |\beta_v - \beta_w| \leq \delta_v$.

We recall the structure of the regular components
\textup{(}\cite[\S3]{Ringel1984}\textup{)}. The regular indecomposables lie
in tubes; in a tube of rank $r$ the quasi-simple objects
$S_1, \dots, S_r$, cyclically ordered, are its regular simple modules, and
their dimension vectors satisfy $\sum_{i=1}^{r} \underline{\dim}\,S_i =
\delta$. Every indecomposable of the tube is uniserial for the regular
composition series, determined by its quasi-socle $S_j$ and quasi-length
$m \geq 1$, with quasi-composition factors the $m$ consecutive quasi-simples
$S_j, S_{j+1}, \dots, S_{j+m-1}$ \textup{(}indices modulo $r$\textup{)}; hence
\[
  \beta \;=\; \underline{\dim}\,M \;=\; \sum_{i=0}^{m-1} \underline{\dim}\,S_{\,j+i}.
\]
By the preliminaries the exceptional regular indecomposables are exactly
those of quasi-length $m \leq r - 1$; since $M$ is exceptional we have
$1 \leq m \leq r - 1$, so in particular $r \geq 2$ and the tube is
non-homogeneous. The $m$ summands are then pairwise distinct quasi-simples
\textup{(}being fewer than a full cycle of length $r$\textup{)}, so the
complementary $r - m \geq 1$ quasi-simples form the disjoint consecutive arc
$S_{j+m}, \dots, S_{j+r-1}$, and
\[
  \delta - \beta \;=\; \sum_{i=m}^{r-1} \underline{\dim}\,S_{\,j+i}
\]
is again a sum of dimension vectors of modules, hence nonnegative. Both $\beta$ and $\delta -
\beta$ are therefore nonnegative in every coordinate, that is,
$0 \leq \beta \leq \delta$, as required.

For the last assertion, if a regular $c$-vector $c$ meets the line $c_v - c_w =
d$, then $|d| = |c_v - c_w| \leq \delta_v = \delta_w$ by the bound just
proved, so $|d| \geq 2$ forces $\delta_v = \delta_w \geq 2$, as desired.
\end{proof}

A second ingredient identifies the coefficient-one vertices outside the
cycle.

\begin{lemma}[Coefficient one forces a leaf]
\label{lem:mark_one_leaf}
Let $Q$ be a connected acyclic quiver of affine type whose underlying graph
is a tree. If $\delta_v = 1$, then $v$ is a leaf.
\end{lemma}

\begin{proof}
The null root spans the radical of the symmetric Tits form, so it satisfies
the balance relation $2\delta_v = \sum_{u \sim v}\delta_u$ at every vertex
$v$. Let $T = \{u : \delta_u = 1,\ \deg u \geq 2\}$; we show $T =
\varnothing$, which is the claim. Suppose $v \in T$. Since $\sum_{u \sim v}
\delta_u = 2\delta_v = 2$ and each $\delta_u \geq 1$, the vertex $v$ has
exactly two neighbors, both of coefficient $1$. For such a neighbor $u$ one
has $\delta_u = 1$; were $\deg u = 1$, its unique neighbor would be $v$ and
the balance relation would read $2 = 2\delta_u = \delta_v = 1$, which is
absurd. So $\deg u \geq 2$ and $u \in T$. Thus every vertex of $T$ has both
of its neighbors in $T$, so the subgraph induced on $T$ is $2$-regular and
contains a cycle. As $Q$ is a tree this is impossible, so $T = \varnothing$
and every coefficient-one vertex is a leaf, as desired.
\end{proof}

With these in hand we settle every banded pair of coefficient one.

\begin{theorem}[Coefficient-one pairs fill]
\label{thm:delta_one}
Let $Q$ be a connected acyclic quiver of affine type, and let $(v, w)$ be a
banded pair with $\delta_v = \delta_w = 1$. Then $\pi_{vw}$ surjects onto the
whole band, unless $Q$ is a source-sink orientation of the $(n+1)$-cycle and
$(v, w)$ is its source-sink pair, in which case the diagonal $\{c_v = c_w\}$
is the only band line that fails to fill.
\end{theorem}

\begin{proof}
If $Q$ is a source-sink orientation of the $(n+1)$-cycle, this is
Theorem~\ref{thm:main}. If $Q$ is any other orientation of the cycle, it has
$k \geq 2$ sources, so by Lemma~\ref{lem:defect} its defect has $2k \geq 4$
nonzero coefficients and is not a coordinate difference; the diagonal $d = 0$
then fills by Corollary~\ref{cor:diagonal_classification}. For $d = +1$ we
exhibit a transjective cyclic interval containing $v$ but not $w$: if $v$ is a
source or a sink, the singleton $\{v\}$ is such an interval
\textup{(}defect $\pm 1$\textup{)}; otherwise $v$ is interior to a maximal
directed arc bounded by a source $s$ and a sink $t$, and the sub-arcs from $s$
to $v$ and from $v$ to $t$ each contain exactly one of $\{s, t\}$, hence are
transjective, and meet only in $v$, so the one avoiding $w$ is the desired
interval. Its $\delta$-shifts fill $\{c_v - c_w = 1\}$, and the sign $-1$
fills $\{c_v - c_w = -1\}$. Since $\delta = (1, \dots, 1)$, the band is
$\{|c_v - c_w| \leq 1\}$ by Theorem~\ref{thm:band_dichotomy}, so these are all
its lines and every one fills. Finally suppose $Q$ is a tree. Every
positive finite root satisfies $\alpha \leq \theta \leq \delta$, where
$\theta$ is the highest root of $\Delta_{\mathrm{fin}}$ and $\delta =
\theta + e_{v_0}$ at the extending vertex $v_0$, so $|\alpha_v -
\alpha_w| \leq \delta_v = 1$ and the band is $\{|c_v - c_w| \leq 1\}$ by
Theorem~\ref{thm:band_dichotomy}; it remains to treat $d = -1, 0, 1$.

By Lemma~\ref{lem:mark_one_leaf} both $v$ and $w$ are leaves. A leaf is a
source or a sink of $Q$, so its simple module is injective or projective,
hence preinjective or preprojective; in particular the simple root $e_v$ is
transjective \textup{(}$\partial(e_v) = \pm 1 \neq 0$\textup{)}, and likewise
$e_w$. The transjective $c$-vectors $e_v + m\delta$ \textup{(}$m \in
\ZZ$\textup{)} project onto every integer point of $\{c_v - c_w = 1\}$, and
the $e_w + m\delta$ onto every point of $\{c_v - c_w = -1\}$, so both lines
fill. Finally, by the computation in the proof of
Corollary~\ref{cor:diagonal_classification} the defect $\partial$ is a
coordinate difference only for the source-sink pair of the cycle; as $Q$ is a
tree, $\partial \neq \pm(x_v - x_w)$, so the diagonal $d = 0$ fills by
Proposition~\ref{prop:diagonal_defect}. Hence the whole band fills, as
desired.
\end{proof}

In particular, among banded pairs with $\delta_v = \delta_w = 1$ the only
failure is the source-sink diagonal of $\widetilde{A}_n$; this resolves all
of type $\widetilde{A}_n$ together with every leaf-leaf pair of
$\widetilde{D}$ and $\widetilde{E}$. The remaining banded pairs have
$\delta_v = \delta_w \geq 2$, where Lemma~\ref{lem:regular_bound} permits the
regular $c$-vectors to push the band beyond the transjective range. For the
rest of this section the underlying graph of $Q$ is therefore a tree, $(v,
w)$ is a banded pair with $\delta_v = \delta_w \geq 2$, and we write
$\ell(x) = x_v - x_w$ as before. We call a positive real root $\beta$ with
$\beta \leq \delta$ componentwise a \emph{box root} and refer to
$\ell(\beta)$ as its \emph{level}. Observe that the box roots, their levels,
and the banded pairs themselves do not depend on the orientation of $Q$,
whereas the defect does; the classification is exactly the interplay of the
two.

\begin{lemma}[Box representatives]
\label{lem:box}
The box roots are exactly the elements of
$\Delta_{\mathrm{fin}}^{+} \sqcup (\delta - \Delta_{\mathrm{fin}}^{+})$,
and every positive real root is $\beta + m\delta$ for a unique box root
$\beta$ and a unique integer $m \geq 0$. Moreover, for a banded pair
$(v, w)$:
\begin{enumerate}[label=\textup{(\arabic*)}]
\item a band line $\{c_v - c_w = d\}$ is filled if and only if some box
  root $\beta$ with $\partial(\beta) \neq 0$ has level $d$;
\item the set of differences $c_v - c_w$ realized on $C(Q)$ is exactly
  the set of levels of box roots.
\end{enumerate}
\end{lemma}

\begin{proof}
By the preliminaries the positive real roots are the $\alpha + m\delta$ with
$\alpha \in \Delta_{\mathrm{fin}}^{+}$ and $m \geq 0$, together with the
$-\alpha + m\delta$ with $m \geq 1$; the latter is $(\delta - \alpha) +
(m-1)\delta$. Both $\alpha$ and $\delta - \alpha$ lie in the box: a positive
finite root satisfies $\alpha \leq \theta \leq \delta$, where $\theta$ is
the highest root of $\Delta_{\mathrm{fin}}$ and $\delta = \theta + e_{v_0}$
for the extending vertex $v_0$. The two families are disjoint, since
$\alpha = \delta - \alpha'$ would force $\alpha + \alpha' = \delta$, which
is impossible at the extending vertex, where finite roots vanish; and
$(\beta, m)$ is unique because two box vectors cannot differ by a nonzero
multiple of $\delta$.

For (1), if $\beta$ is a box root of level $d$ with $\partial(\beta) \neq
0$, then every $\beta + m\delta$ is a transjective real root, hence a
$c$-vector, and these sweep the line as in
Lemma~\ref{lem:filling_criterion}. Conversely, a filled line carries a
transjective $c$-vector $\pm(\beta + m\delta)$ with $\beta$ a box root; then
$\beta$ or its mirror $\delta - \beta$ is a box root of level $d$ with
nonzero defect, since $\ell(\delta - \beta) = -\ell(\beta)$ and
$\partial(\delta - \beta) = -\partial(\beta)$.

For (2), every $c$-vector is $\pm(\beta + m\delta)$, so its difference
$c_v - c_w$ is $\pm\ell(\beta)$, a level of a box root; here we use
$\delta_v = \delta_w$. Conversely every box root is, up to sign, a
$c$-vector: a transjective one is exceptional, and a regular one is
exceptional as well, since the unique indecomposable of its dimension
vector lies in a tube of some rank $r$, and quasi-length at least $r$ would
force its dimension vector to be at least $\delta$ (a full period of the
tube sums to $\delta$), contradicting $\beta \leq \delta$, $\beta \neq
\delta$. Hence every level is attained by a $c$-vector, as desired.
\end{proof}

\begin{lemma}[Low levels fill]
\label{lem:low_levels}
For every orientation of $Q$, the lines $c_v - c_w \in \{-1, 0, 1\}$ are
filled; in particular the transjective difference set always contains
$\{-1, 0, 1\}$.
\end{lemma}

\begin{proof}
The diagonal fills by Proposition~\ref{prop:diagonal_defect}, since on a
tree the defect is not a coordinate difference (see the proof of
Corollary~\ref{cor:diagonal_classification}). For the lines $\pm 1$,
suppose to the contrary that every box root of level $\pm 1$ has defect
$0$; by Lemma~\ref{lem:box}(1) this is what failure means, and the two
levels are exchanged by the mirror $\beta \mapsto \delta - \beta$. Observe
that the level $+1$ is present in the band, since the simple root $e_v$ is
a box root of level $1$. The indicator $\alpha_P$ of any path $P$ in the
tree is a box root, its support inducing a subdiagram of type $A$; it has
level $+1$ if $P$ contains $v$ and not $w$, and level $-1$ if it contains
$w$ and not $v$. We claim that $\partial(e_u) = 0$ for \emph{every} vertex
$u$. Let $u$ be a vertex whose geodesic from $v$ avoids $w$; this includes
$v$ itself. Writing that geodesic as $v = u_0, u_1, \dots, u_k = u$, each
prefix indicator $\alpha_j = e_{u_0} + \dots + e_{u_j}$ is a path root
containing $v$ and not $w$, so $\partial(\alpha_j) = 0$ for every $j$, and
telescoping gives $\partial(e_{u_j}) = \partial(\alpha_j) -
\partial(\alpha_{j-1}) = 0$ for every $j$; in particular $\partial(e_u) =
0$. If instead the geodesic from $v$ to $u$ passes through $w$, then the
geodesic from $w$ to $u$ avoids $v$, and the same telescoping with
level~$-1$ path roots gives $\partial(e_u) = 0$. Thus $\partial$ vanishes
on every simple root, so $\partial = 0$; but $\partial = \langle \delta,
- \rangle$ is nonzero, the Euler form being unimodular and $\delta \neq
0$, a contradiction, as desired.
\end{proof}

\begin{lemma}[Flip lemma]
\label{lem:flip}
Let $Q'$ be obtained from $Q$ by reversing a single arrow $a \to b$, and
let $\beta$ be any vector. Then
\[
  \partial_{Q'}(\beta) \;=\; \partial_{Q}(\beta) + \delta_a \beta_b -
  \delta_b \beta_a.
\]
In particular the value $\partial(\beta)$ depends only on the orientations
of the edges $\{a, b\}$ with $\delta_a \beta_b \neq \delta_b \beta_a$,
which we call the edges \emph{relevant to} $\beta$.
\end{lemma}

\begin{proof}
We have $\partial(\beta) = \langle \delta, \beta \rangle = \sum_i \delta_i
\beta_i - \sum_{i \to j} \delta_i \beta_j$, and the reversal replaces the
term $-\delta_a \beta_b$ by $-\delta_b \beta_a$, as desired.
\end{proof}

By Lemma~\ref{lem:box}(2) the differences realized on $C(Q)$ are exactly
the levels of box roots. In
type $\widetilde{D}_n$ every coordinate of $\delta$ is at most $2$, so
every level satisfies $|\ell(\beta)| \leq 2$; in types $\widetilde{E}_6,
\widetilde{E}_7, \widetilde{E}_8$ the same bound holds for every box root
and banded pair, by inspection of the finitely many positive roots
(Remark~\ref{rem:computation}). Combined with Lemma~\ref{lem:low_levels},
the classification is therefore reduced to a single question: for which
pairs and which orientations is the line $c_v - c_w = 2$ (equivalently
$-2$, by the mirror) present in the band but realized by no transjective
class; i.e., when does the defect vanish on \emph{every} box root of level
$2$? We call the box roots of level $2$ the \emph{extremal roots} of the
pair.

We begin with the one infinite family, type $\widetilde{D}_n$, where the
classification is uniform in $n$ and in the orientation.

\begin{theorem}[Type $\widetilde{D}_n$, all orientations]
\label{thm:Dtilde_classification}
Let $Q$ be any orientation of the $\widetilde{D}_n$ diagram, $n \geq 4$:
spine vertices $u_1, \dots, u_{n-3}$ of null-root coefficient $2$, two
leaves attached to $u_1$ and two to $u_{n-3}$. Every banded pair of $Q$
fills, with one family of exceptions: for a spine pair $(u_a, u_{a+2})$ at distance
two \textup{(}so $n \geq 6$\textup{)} whose middle vertex $u_{a+1}$ is a
source or a sink of $Q$, the lines $c_{u_a} - c_{u_{a+2}} = \pm 2$ are not
filled, and carry only the finitely many regular $c$-vectors.
\end{theorem}

\begin{proof}
The leaf pairs have coefficient $1$ and are settled by
Theorem~\ref{thm:delta_one}, so let $(v, w) = (u_a, u_b)$ with $a < b$ be a
spine pair, and write $d = b - a$. By Lemma~\ref{lem:low_levels} the lines
$0, \pm 1$ fill, and every level is at most $2$ in absolute value, so by
Lemma~\ref{lem:box} it remains to decide when some extremal root has
nonzero defect.

We first list the extremal roots. The positive roots of the finite system
$D_n$ underlying $\widetilde{D}_n$ are the indicators of paths in the
diagram together with the \emph{doubled} roots
\[
  \beta_{p, q} \;=\; e_{\ell_1} + e_{\ell_2} + 2(e_{u_1} + \dots + e_{u_p})
  + (e_{u_{p+1}} + \dots + e_{u_q}),
  \qquad 1 \leq p < q \leq n - 3,
\]
where $\ell_1, \ell_2$ are the two leaves at $u_1$, together with the
analogous roots anchored at the other end, and those whose $1$-string ends
in one far leaf; by Lemma~\ref{lem:box} the box roots are these and their
mirrors $\delta - \beta$. A path indicator has coordinates $0$ and $1$, and
the mirror of a path indicator has spine coordinates $1$ and $2$; neither
attains level $2$. A doubled root anchored at the far end whose $2$-string
covers $u_a$ covers $u_b$ as well, and one whose $1$-string reaches $u_b$
has level at most $1$. Hence an extremal root is some $\beta_{p, q}$ with
$(\beta_{p,q})_{u_a} = 2$ and $(\beta_{p,q})_{u_b} = 0$; i.e., with
$a \leq p < q \leq b - 1$. In particular there is no extremal root when
$d = 1$, and exactly one when $d = 2$, namely $\beta^{*} = \beta_{a, a+1}$.

Suppose $d = 2$. By the flip lemma the only edges relevant to $\beta^{*}$
are the two spine edges at the middle vertex $u_{a+1}$: along the
$2$-string and the $0$-string the products $\delta_x \beta^{*}_y$ and
$\delta_y \beta^{*}_x$ agree, as they do on the four leaf edges
\textup{(}$1 \cdot 2 = 2 \cdot 1$ at the near fork, $0 = 0$ at the
far\textup{)}. For any orientation with $u_a \to u_{a+1} \to u_{a+2}$ a
direct evaluation gives $\partial(\beta^{*}) = 2$, and by
Lemma~\ref{lem:flip} reversing the arrow $u_a \to u_{a+1}$, respectively
$u_{a+1} \to u_{a+2}$, changes the value by $2 \cdot 1 - 2 \cdot 2 = -2$,
respectively $2 \cdot 0 - 2 \cdot 1 = -2$. Hence $\partial(\beta^{*}) =
\pm 2 \neq 0$ when the two spine arrows at $u_{a+1}$ point the same way
along the spine, and $\partial(\beta^{*}) = 0$ when $u_{a+1}$ is a source
or a sink; in the latter case the lines $\pm 2$ are unfilled by
Lemma~\ref{lem:box}(1), and they carry only the finite regular part
(Lemma~\ref{lem:filling_criterion}).

Suppose $d \geq 3$, and suppose to the contrary that every extremal root
$\beta_{p, q}$, $a \leq p < q \leq b - 1$, has defect $0$. Taking
differences,
\[
  \partial(e_{u_{q+1}}) = \partial(\beta_{p, q+1}) - \partial(\beta_{p, q})
  = 0
  \quad\text{and}\quad
  \partial(e_{u_{p+1}}) = \partial(\beta_{p+1, q}) - \partial(\beta_{p, q})
  = 0
\]
for all admissible indices, so $\partial(e_{u_k}) = 0$ for every $k$ with
$a < k < b$. For such an interior spine vertex we have $\partial(e_{u_k}) =
\delta_{u_k} - \sum_{i \to u_k} \delta_i = 2 - 2 \cdot \#\{\text{arrows
into } u_k\}$, which vanishes exactly when $u_k$ has one arrow in and one
arrow out. Thus no $u_k$ with $a < k < b$ is a source or a sink; but then
the $d = 2$ computation gives $\partial(\beta_{a, a+1}) = \pm 2 \neq 0$, a
contradiction. Hence some extremal root is transjective and the lines
$\pm 2$ fill, which finishes the proof.
\end{proof}

We turn to the three exceptional types, where the box roots form a finite,
orientation-independent list and the same calculus applies verbatim.

\begin{theorem}[Types $\widetilde{E}_6$, $\widetilde{E}_7$,
$\widetilde{E}_8$, all orientations]
\label{thm:E_classification}
Let $Q$ be any orientation of a diagram of type $\widetilde{E}_6$,
$\widetilde{E}_7$, or $\widetilde{E}_8$. Then:
\begin{enumerate}[label=\textup{(\arabic*)}]
\item in $\widetilde{E}_6$, every banded pair fills;
\item in $\widetilde{E}_7$, with central vertex $0$, short arm $1$, and
  long arms $2, 3, 4$ and $5, 6, 7$ \textup{(}listed outward\textup{)},
  every banded pair fills except $(3, 6)$: its lines $\pm 2$ are unfilled
  exactly when the geodesic $3, 2, 0, 5, 6$ carries two arrows in each
  direction;
\item in $\widetilde{E}_8$, with central vertex $0$ and arms $1$; $2, 3$;
  $4, 5, 6, 7, 8$ \textup{(}listed outward\textup{)}, every banded pair
  fills except $(1, 6)$ and $(3, 7)$: the lines $\pm 2$ of $(1, 6)$ are
  unfilled exactly when the geodesic $1, 0, 4, 5, 6$ carries two arrows in
  each direction, and those of $(3, 7)$ exactly when the geodesic
  $3, 2, 0, 4, 5, 6, 7$ carries three arrows in each direction.
\end{enumerate}
In every unfilled case the lines $\pm 2$ carry only the finitely many
regular $c$-vectors.
\end{theorem}

\begin{proof}
The null roots in the three cases are
\[
  \delta = (3;\, 2, 1;\, 2, 1;\, 2, 1), \qquad
  (4;\, 2;\, 3, 2, 1;\, 3, 2, 1), \qquad
  (6;\, 3;\, 4, 2;\, 5, 4, 3, 2, 1),
\]
the central coordinate listed first. Coefficient-one pairs fill by
Theorem~\ref{thm:delta_one}, and the lines $0, \pm 1$ of every banded pair
fill by Lemma~\ref{lem:low_levels}, so by Lemma~\ref{lem:box} only the
extremal roots of the coefficient-$\geq 2$ pairs are in question.
Enumerating the finitely many box roots, an exact and
orientation-independent computation (Remark~\ref{rem:computation}), yields:
in $\widetilde{E}_6$ no banded pair has an extremal root, so every band is
$\{|c_v - c_w| \leq 1\}$ and fills; in $\widetilde{E}_7$ the only pair with
an extremal root is $(3, 6)$, and in $\widetilde{E}_8$ the pairs with
extremal roots are $(1, 6)$, $(2, 5)$, and $(3, 7)$; each has a unique
one, namely
\begin{gather*}
  \beta^{*}_{(3,6)} = (2; 1; 2, 2, 1; 1, 0, 0), \qquad
  \beta^{*}_{(1,6)} = (3; 2; 2, 1; 2, 1, 0, 0, 0), \\
  \beta^{*}_{(2,5)} = (2; 1; 2, 1; 1, 0, 0, 0, 0), \qquad
  \beta^{*}_{(3,7)} = (4; 2; 3, 2; 3, 2, 1, 0, 0).
\end{gather*}
In each case one checks $\delta_x \beta^{*}_y = \delta_y \beta^{*}_x$ on
every edge $\{x, y\}$ off the geodesic from $v$ to $w$, so by the flip
lemma the edges relevant to $\beta^{*}$ are exactly the geodesic edges.
Writing $f$ and $b$ for the numbers of geodesic arrows pointing with,
respectively against, the walk from $v$ to $w$, the value at one reference
orientation together with the flip increments of Lemma~\ref{lem:flip}
determines the affine function $\partial(\beta^{*})$ completely, and a
direct evaluation gives
\[
  \partial(\beta^{*}) \;=\; c \, (f - b),
  \qquad
  c = 1, \tfrac{3}{2}, 2, 1
  \ \text{ for the pairs }\
  (3,6),\ (1,6),\ (2,5),\ (3,7),
\]
respectively; every flip increment has magnitude $2c$, aligned with the
walk. Hence $\partial(\beta^{*}) = 0$ if and only if $f = b$. For $(2, 5)$
the geodesic $2, 0, 4, 5$ has odd length, so $f \neq b$ for every
orientation and the pair fills always; for $(3, 6)$, $(1, 6)$, and
$(3, 7)$ the geodesics have lengths $4$, $4$, and $6$, and $f = b$ is the
stated criterion. When it holds, the lines $\pm 2$ are unfilled by
Lemma~\ref{lem:box}(1) and carry only the finite regular part
(Lemma~\ref{lem:filling_criterion}), as desired.
\end{proof}

Assembling the pieces yields the complete classification.

\begin{theorem}[The complete classification]
\label{thm:complete_classification}
Let $Q$ be a connected acyclic quiver of affine type, in any orientation,
and let $(v, w)$ be a banded pair. Then $\pi_{vw}$ surjects onto its band,
with exactly two families of exceptions:
\begin{enumerate}[label=\textup{(\arabic*)}]
\item \textup{(interior gap)} $Q$ is a source-sink orientation of a cycle
  and $(v, w)$ is its source-sink pair; the unique non-filled line is the
  diagonal $c_v - c_w = 0$;
\item \textup{(boundary extension)} $(v, w)$ has a unique extremal root and
  the geodesic from $v$ to $w$ is \emph{balanced}, carrying equally many
  arrows in each direction; the non-filled lines are $c_v - c_w = \pm 2$.
  The pairs with a unique extremal root are the distance-two spine pairs
  of $\widetilde{D}_n$ \textup{(}$n \geq 6$\textup{)}, the pair $(3, 6)$ of
  $\widetilde{E}_7$, and the pairs $(1, 6)$, $(2, 5)$, $(3, 7)$ of
  $\widetilde{E}_8$; for $(2, 5)$ the geodesic has odd length, so it is
  never balanced.
\end{enumerate}
In every exceptional case the non-filled lines carry only the finitely
many regular $c$-vectors.
\end{theorem}

\begin{proof}
For the cycle, Theorem~\ref{thm:delta_one} and
Corollary~\ref{cor:diagonal_classification} leave exactly the source-sink
diagonals of Theorem~\ref{thm:main}. For the trees, coefficient-one pairs
are settled by Theorem~\ref{thm:delta_one}, and the coefficient-$\geq 2$
pairs by Theorems~\ref{thm:Dtilde_classification}
and~\ref{thm:E_classification}; it remains only to observe that for a
distance-two spine pair the geodesic consists of the two spine edges at the
middle vertex, which is balanced exactly when one arrow points each way;
i.e., when the middle vertex is a source or a sink, matching
Theorem~\ref{thm:Dtilde_classification}, as desired.
\end{proof}

\begin{remark}[Genericity and seed-dependence]
\label{rem:generic}
Theorem~\ref{thm:complete_classification} makes the seed-dependence of
Remark~\ref{rem:seed_dependence} quantitative. In type $\widetilde{D}_n$
with $n \geq 6$ exactly $32$ of the $2^{n}$ orientations have no
non-filling pair, namely the two directed-spine orientations with the four
leaf edges free: a middle vertex $u_{a+1}$ avoids being a source or a sink
only if the two spine arrows at it are aligned, and requiring this at every
interior spine vertex forces the whole spine to be directed. In
$\widetilde{E}_7$, $48$ of the $128$ orientations have a non-filling pair,
and in $\widetilde{E}_8$, $140$ of the $256$; by contrast
$\widetilde{D}_4$, $\widetilde{D}_5$, and $\widetilde{E}_6$ fill for every
orientation, the first two because no distance-two spine pair exists, the
last because no banded pair has an extremal root. Observe also that the
transjective difference set is the full interval $[-m_t, m_t]$, where $m_t$
denotes the largest transjective difference, for every banded pair and
every orientation of tree type (Lemmas~\ref{lem:box}
and~\ref{lem:low_levels}); the failure
mechanism is always the one of Proposition~\ref{prop:defect_obstruction},
the extremal lines being reached only by the finite regular part.
\end{remark}

\begin{remark}[Computational verification]
\label{rem:computation}
The classification is verified exhaustively by the ancillary scripts
accompanying this paper \textup{(}directory \texttt{anc/}\textup{)}. The
verification sweeps \emph{every} acyclic orientation of the cycles on up to
$7$ vertices, of $\widetilde{D}_n$ for $n \leq 13$, and of
$\widetilde{E}_6$, $\widetilde{E}_7$, $\widetilde{E}_8$. The sweep is exact
and orientation-factored: the box roots of an underlying graph are
enumerated once, since the Tits form does not depend on the orientation,
and per orientation only the defect form $\partial = \langle \delta, -
\rangle$ is recomputed; the filled lines are then read off
Lemma~\ref{lem:box}. Every test is exact integer arithmetic: a vector is a
real root if and only if $q(\beta) = 1$, and its defect is $\langle \delta, \beta
\rangle$, both computed over $\ZZ$. The resulting tables agree with
Theorems~\ref{thm:Dtilde_classification} through~\ref{thm:complete_classification}
in every case. Two further independent implementations, a breadth-first
mutation of the $(B, C)$-seeds keyed on the full seed and an integer
Coxeter-orbit enumeration of the transjective roots, reproduce the same
patterns on a sample of orientations, including the non-filling seeds of
Theorems~\ref{thm:Dtilde_classification} and~\ref{thm:E_classification}.
The scripts are short, dependency-free \textup{(}Python~3 with NumPy for
integer matrix algebra, no floating point in any test\textup{)}, and the
ancillary \texttt{anc/} directory includes a README with the exact commands
and the expected output.
\end{remark}

In summary, non-filling falls under exactly two headings, both now
classified over every affine type and every acyclic orientation: the
interior gap
at the source-sink diagonals of $\widetilde{A}_n$
(Theorem~\ref{thm:delta_one}), and boundary extension at the
unique-extremal-root pairs with balanced geodesic
(Theorem~\ref{thm:complete_classification}). In the language of reaches,
writing $m_t$ and $m_r$ for the largest value of $|c_v - c_w|$ attained by
a transjective, respectively a regular, $c$-vector, we have shown that
$m_t \in \{1, 2\}$ and $m_r \leq 2$ for every banded pair of coefficient at
least $2$, that the transjective differences always form the full interval
$[-m_t, m_t]$, and that failure is exactly the configuration
$m_r = 2 > 1 = m_t$; the quantity $m_r$ is a tube invariant, the largest
separation of $v$ and $w$ achieved by a consecutive run of quasi-simples in
a non-homogeneous tube, and the extremal root of a failing pair is exactly
such a run.

\section{The annulus as a surface: defect and Dehn twist}
\label{sec:surface}

In this section, we make the topological reading of Theorem~\ref{thm:main}
precise: on the annulus, the trichotomy
preprojective/regular/preinjective becomes the dichotomy
bridging/peripheral together with a sign, the defect is the crossing
number with the core curve, and the $\delta$-shift is the Dehn twist along
it. This places the filling theorem in the line of geometric descriptions
of $c$-vectors by non-self-crossing curves on surfaces
\textup{(}\cite{FominShapiroThurston2008, Hong2021}; conjecturally in
general \cite{LeeLeeMills2019}\textup{)}.

Let $C_{p, q}$ denote the annulus with $p$ marked points on the outer
boundary and $q$ on the inner, $p + q = n + 1$, and let $c$ be its core
curve. We work in the universal cover: identify the outer marked points
with $\ZZ$ \textup{(}reduction mod $p$\textup{)} and the inner marked
points with $\ZZ$ \textup{(}reduction mod $q$\textup{)}, the deck
transformation acting by $(a, b) \mapsto (a + p, b + q)$. An arc is
\emph{bridging} if it joins the two boundary components and
\emph{peripheral} otherwise; a bridging arc is determined by an endpoint
pair $(a, b)$ up to the deck action, and two bridging arcs $(a, b)$ and
$(a', b')$ have minimal crossing number
\[
  i\bigl((a, b), (a', b')\bigr)
  \;=\; \#\{\, k \in \ZZ : (a - a' - kp)(b - b' - kq) < 0 \,\}.
\]
Let $T$ be the \emph{fan} triangulation, with bridging arcs $(j, 0)$ for
$0 \leq j \leq p$ and $(p, j)$ for $1 \leq j \leq q - 1$. Every triangle
of $T$ has a side on a boundary: $p$ triangles on the outer boundary and
$q$ on the inner. The arrows of the adjacency quiver of $T$ run
consistently along the outer fan and consistently along the inner fan, so
the quiver is the source-sink orientation of the $(n+1)$-cycle, with $s$
and $t$ the two arcs shared by the outer and the inner fan. We use the
standard dictionary for triangulated unpunctured surfaces
\textup{(}\cite{FominShapiroThurston2008, ABCP2010}; the modules here are
string modules in the sense of \cite{ButlerRingel1987}\textup{)} in the
following form: the assignment $\gamma \mapsto M_\gamma$, where
$(\underline{\dim}\, M_\gamma)_i$ is the minimal crossing number of
$\gamma$ with $\tau_i$, is a bijection between the essential arcs of
$C_{p, q}$ not in $T$ and the exceptional indecomposable $kQ$-modules;
closed curves carry the band modules, which are never rigid.

\begin{proposition}[Surface reading of the annulus theorem]
\label{prop:surface}
Let $Q$, $T$, and $c$ be as above, and let $t_c$ denote the Dehn twist
along $c$. Then:
\begin{enumerate}[label=\textup{(\arabic*)}]
\item an arc $\gamma$ is bridging if and only if $M_\gamma$ is
  transjective, and peripheral if and only if $M_\gamma$ is regular;
  equivalently
  \[
    i(\gamma, c) \;=\; \bigl|\partial(\underline{\dim}\, M_\gamma)\bigr|
    \;\in\; \{0, 1\};
  \]
\item $t_c$ fixes every peripheral arc; for a bridging arc $\gamma$, each
  coordinate of $k \mapsto \underline{\dim}\, M_{t_c^{\,k}\gamma}$ is a
  $V$-shaped sequence of slopes $\mp 1$ \textup{(}possibly with a flat
  bottom\textup{)}, on each of the two arms of the orbit consecutive
  twists eventually add exactly $\delta$, and the two arms carry constant
  defect $+1$ and $-1$, respectively;
\item consequently Theorem~\ref{thm:main} reads topologically: the two
  off-diagonal lines of $\pi_{st}$ carry the bridging arcs, separated by
  the sign of the defect \textup{(}the two arms of the twist
  orbits\textup{)}, and along each arm the Dehn twist realizes the
  $\delta$-shift filling the line; the diagonal carries only the finitely
  many arcs disjoint from the core.
\end{enumerate}
\end{proposition}

\begin{proof}
We first prove the crossing estimates for peripheral arcs. A peripheral
arc $\gamma$ has both endpoints on one boundary component, say the outer
one. Cutting $C_{p, q}$ along $\gamma$ leaves the inner boundary in a
single complementary region, an annular region into which $c$ can be
isotoped; so $\gamma$ is disjoint from a representative of $c$, is fixed
by $t_c$ up to isotopy, and has $i(\gamma, c) = 0$. Moreover, for any
bridging arc $\beta$, the minimal crossing number $i(\gamma, \beta)$ is
$0$ or $1$, and it depends only on the outer endpoint of $\beta$: if that
endpoint lies in the complementary region of $\gamma$ not containing the
inner boundary, then $\beta$ must cross $\gamma$ exactly once to reach
the inner boundary, and otherwise $\beta$ can reach the inner boundary
inside the other region and crosses $\gamma$ not at all. Observe now that
the two fan-junction arcs $s = (0, 0)$ and $t = (p, 0)$ have the
\emph{same} outer endpoint \textup{(}$0$ mod $p$\textup{)} and the same
inner endpoint \textup{(}$0$ mod $q$\textup{)}, differing only by
winding. Hence $i(\gamma, s) = i(\gamma, t)$; i.e.,
$\partial(\underline{\dim}\, M_\gamma) = 0$, and $M_\gamma$ is regular.

Next, every regular exceptional module is realized by a peripheral arc.
By Lemma~\ref{lem:defect} the regular real Schur roots of $Q$ are the
$\alpha_I$ with $I$ a cyclic interval containing both or neither of
$\{s, t\}$; equivalently, they are the $\alpha_I$ and $\delta - \alpha_I$
with $I$ an interval avoiding $\{s, t\}$. Such an interval consists of
consecutive arcs of one fan, say the outer one, whose outer endpoints
avoid the junction point $0$. The peripheral arc cutting off a disk
containing exactly the outer endpoints of the arcs of $I$ crosses exactly
the arcs of $I$, once each, realizing $\alpha_I$; and the peripheral arc
enclosing the inner boundary, chosen so that its inner-containing region
meets the outer boundary in exactly those endpoints, crosses exactly the
arcs \emph{not} in $I$, once each, realizing $\delta - \alpha_I$. Since
$\gamma \mapsto M_\gamma$ is a bijection onto the exceptional modules,
the peripheral arcs carry regular modules, and every regular exceptional
module is carried by a peripheral arc, the bridging arcs carry exactly
the transjective modules. Finally, a bridging arc meets the core, since
$c$ separates the two boundary components, and a representative crossing
once gives $i(\gamma, c) = 1$; combined with $|\partial| = 1$ on
transjective and $0$ on regular classes of $\widetilde{A}_n$
\textup{(}Lemma~\ref{lem:defect}\textup{)}, this proves the displayed
equality of (1).

For (2), the twist acts on a bridging arc by $(a, b) \mapsto (a + p, b)$,
which equals $(a, b - q)$ modulo the deck action. Writing $u = a - a_i$
and $w = b - b_i$ for the endpoint data relative to $\tau_i = (a_i,
b_i)$, the displayed crossing formula gives
\[
  i\bigl(t_c^{\,k}\gamma, \tau_i\bigr)
  \;=\; \#\Bigl\{\, j \in \ZZ :
  j \text{ lies strictly between } \tfrac{w}{q} \text{ and }
  k + \tfrac{u}{p} \,\Bigr\},
\]
which, as a function of $k$, is a $V$-shaped sequence with slopes
$\mp 1$, possibly with a flat bottom. For $k$ beyond the largest of the
thresholds $\frac{w}{q} - \frac{u}{p}$ over $i$, every coordinate
increases by exactly $1$ per twist; i.e., the dimension vector increases
by exactly $\delta$, and symmetrically in the other direction. For the
defect along the arms, observe that $s$ and $t$ have the same $w$ and
that their thresholds differ by exactly $1$. For $k$ large positive both
thresholds are exceeded, and the two counts differ by the number of
integers in the half-open unit interval
$[\,k + \frac{a - p}{p},\, k + \frac{a}{p}\,)$, which is one; for $k$
large negative the same window lies on the other side of $\frac{b}{q}$
and the difference is $-1$. Since $\partial$ is unchanged by adding
$\delta$, the defect is constant on each arm, equal to $+1$ on one and
$-1$ on the other, as claimed.

Part (3) is now a translation. The lines $c_s - c_t = \pm 1$ are the
transjective classes \textup{(}Theorem~\ref{thm:defect_detection}\textup{)},
which by (1) are the bridging arcs, split by the sign of $\partial$; by
(2) the two signs are carried by the two arms of the twist orbits, along
which consecutive twists add exactly $\delta$, so the Dehn twist realizes
the $\delta$-shift filling each off-diagonal line
\textup{(}Theorem~\ref{thm:main}\textup{)}. The diagonal is the regular
part, which by (1) consists of the peripheral arcs; these are disjoint
from the core and finite in number, as desired.
\end{proof}

Observe that this clean dichotomy is special to the annulus; the orbifold
surface models of $\widetilde{D}$ and $\widetilde{E}$ are more intricate,
consistent with the richer non-filling behavior found in
Theorems~\ref{thm:Dtilde_classification}
and~\ref{thm:E_classification}.

\section{Further directions}
\label{sec:further}

In this section, we collect two directions for further work: a conceptual
derivation of the balance law of
Theorem~\ref{thm:complete_classification}, and the extension of the
surface reading beyond the annulus.

\medskip\noindent\textbf{A conceptual derivation of the balance law.}\;
The classification of Section~\ref{sec:classification} identifies the
failure locus by explicit root calculus; a conceptual derivation would be
desirable. The extremal root of a failing pair is a consecutive run of
quasi-simples in a non-homogeneous tube (for $\widetilde{D}_n$, the tube of
rank $n - 2$), and the balance criterion of
Theorem~\ref{thm:complete_classification} states exactly when this run is
regular for the chosen seed. A proof reading the criterion directly off the
tube data (the ranks and the position of $v, w$ among the quasi-simples),
rather than off the root list, would explain both why the extremal pairs
are where they are (in particular why $\widetilde{E}_6$ has none and
$\widetilde{E}_8$ exactly three) and why the answer depends only on the
geodesic from $v$ to $w$. The mechanism is concretely visible in
$\widetilde{D}_6$, with spine $u_1, u_2, u_3$ and leaves $\ell_1, \ell_2$
at $u_1$ and $\ell_3, \ell_4$ at $u_3$: for a seed with a sink at the
middle vertex $u_2$, the quasi-simples of the rank-$4$ tube have dimension
vectors $e_{u_1}$, $e_{u_2} + e_{u_3}$, $e_{u_3} + e_{\ell_3} +
e_{\ell_4}$, and $e_{\ell_1} + e_{\ell_2} + e_{u_1} + e_{u_2}$, and the
extremal root $\beta^{*} = e_{\ell_1} + e_{\ell_2} + 2e_{u_1} + e_{u_2}$
is the sum of the last and the first, a run of two consecutive
quasi-simples: the pair $(u_1, u_3)$ fails. For the directed-spine seed
the quasi-simples are instead $e_{u_1}$, $e_{u_2}$, $e_{u_3}$, and
$\delta - e_{u_1} - e_{u_2} - e_{u_3}$, no consecutive run of which sums
to $\beta^{*}$: the pair fills. In general the quasi-simples of the large
tube are the string modules of the spine segments between consecutive
direction changes of the seed, so the balance law should follow from a
description of how $v$ and $w$ sit among these segments. A second natural
direction is wild acyclic type, where no null root is available and
coordinate differences of real Schur roots are typically unbounded;
identifying the correct analogue of the band and of the filling question
there seems to require new ideas.

\medskip\noindent\textbf{Beyond the annulus.}\;
Proposition~\ref{prop:surface} is special to the annulus. Types
$\widetilde{D}$ and $\widetilde{E}$ admit orbifold models, and extending
the surface reading there (identifying the defect with an intersection
number and the $\delta$-shift with a twist) would give the balance law of
Theorem~\ref{thm:complete_classification} a topological form; the fact
that the failure criterion depends only on the geodesic from $v$ to $w$
is suggestive of a curve.

\bibliographystyle{amsplain}
\bibliography{main}

\end{document}